\documentclass[11pt, a4paper]{article}
\usepackage{amsfonts,amssymb,amsmath,amscd,latexsym,makeidx}
\usepackage{amsthm}
\usepackage{hyperref}
\usepackage{color}
\usepackage{comment}
\usepackage[usenames]{xcolor}

\numberwithin{equation}{section}

\newtheorem{thm}{Theorem}[section]
\newtheorem{lem}[thm]{Lemma}
\newtheorem{prop}[thm]{Proposition}

\newtheorem{defn}[thm]{Definition}
\theoremstyle{remark}
\newtheorem{rem}[thm]{Remark}

\DeclareMathOperator{\loc}{loc}
\newcommand{\D}{\Delta}
\newcommand{\vp}{\varphi}
\newcommand{\R}{\mathbb{R}}
\newcommand{\ve}{\varepsilon}

\newcommand{\B}{\mathbb{B}}
\newcommand{\K}{\mathcal{K}}
\newcommand{\lp}{\left(}
\newcommand{\rp}{\right)}
\newcommand{\vol}{\mathrm{vol}}

\title{Concentration phenomena for the fractional $Q$-curvature equation in dimension 3  and fractional Poisson formulas}
\author{Azahara  DelaTorre  \\ {\small  Albert-Ludwigs-Universit\"at Freiburg}\and Maria  del  Mar  Gonzalez\\ \small{Universidad Aut\'onoma de Madrid}\and  Ali  Hyder\\ {\small University of British Columbia Vancouver} \and Luca  Martinazzi \\  {\small Universit\`a di Padova}}

\begin{document}

\maketitle

\begin{abstract}
We study the compactness properties of metrics of prescribed fractional $Q$-curvature of order $3$ in $\R^3$. We will use an approach inspired from conformal geometry, seeing a metric on a subset of $\R^3$ as the restriction of a metric on $\R^4_+$ with vanishing fourth-order $Q$-curvature. We will show that a sequence of such metrics with uniformly bounded fractional $Q$-curvature can blow up on a large set (roughly, the zero set of the trace of a nonpositive biharmonic function $\Phi$ in $\R^4_+$), in analogy with a $4$-dimensional result of Adimurthi-Robert-Struwe, and construct examples of such behaviour. In doing so, we produce general Poisson-type representation formulas (also for higher dimension), which are of independent interest.
\end{abstract}

\section{Introduction}

Consider a Riemannian manifold $(M,g)$. A classical problem in differential geometry is to conformally transform the metric $g$ in such a way that the scalar curvature of the new metric coincides with a given function $K$. When $(M,g)$ is the round sphere, this corresponds to the intensely studied Nirenberg problem, or when $K$ is chosen to be constant we have a so-called uniformization problem.

Similar problems arise and have been (and are being) studied with the Riemannian scalar curvature replaced by other notions of curvature, among which, the $Q$-curvature. For instance, the uniformization problem for closed manifolds of even dimension $2m\ge 4$ has been addressed in \cite{DMalchiodi,Ndi}, under the assumption that the total $Q$-curvature $\int_M Q_g \,\mathrm{dvol_g}$ is not a multiple of the constant $\Lambda_1:=(2m-1)!\,\mathrm{vol}(\mathbb S^{2m})$, which is the total $Q$-curvature of $\mathbb S^{2m}$. Removing this assumption, the problem is still open, to the best of our knowledge.
A fundamental tool in approaching this, and other prescribed curvature problems, is the so-called blow-up (or concentration) analysis of a sequence of metrics with prescribed curvature. For instance, in the seminal paper \cite{Brezis-Merle} Br\'ezis and Merle studied the case of the Gaussian curvature in dimension $2$:

\begin{thm}[\cite{Brezis-Merle}]\label{thbm}
Given an open subset $\Omega$ of $\R^2$, assume that $(u_k)\subset L^1_{\loc}(\Omega)$ is a sequence of weak solutions to
\begin{equation}\label{eqBM}
-\Delta u_k=K_ke^{2u_k} \quad \text{in }\Omega
\end{equation}
with $K_k\ge 0$ and such that $\|K_k\|_{L^\infty(\Omega)}\le \bar \kappa$ and $\|e^{2u_k}\|_{L^1(\Omega)}\le \bar A,$ for $\bar{\kappa},\bar A $ positive constants.
Then, up to subsequences, either
\begin{enumerate}
\item $u_k$ is bounded in $L^{\infty}_{\loc}(\Omega)$, or\par
\item there is a finite (possibly empty) set $B=\{x_1,\ldots,x_N\}\subset\Omega$  (the blow-up set) such that
$u_k(x)\to -\infty$ locally uniformly in $\Omega\setminus B$, and
\begin{equation}\label{conc2d}
K_k e^{2u_k}\stackrel{*}{\rightharpoonup} \sum_{i=1}^N\alpha_i\delta_{x_i} \quad\text{for some numbers }\alpha_i\ge 2\pi,
\end{equation}
where $\stackrel{*}{\rightharpoonup}$ denotes the weak-* convergence in the sense of Radon measures.
\end{enumerate}
\end{thm}
In \eqref{eqBM} the function $K_k$ is the Gaussian curvature of the metric $e^{2u_k}|dx|^2$, having area $\|e^{2u_k}dx\|_{L^1(\Omega)}\le \bar A$. The constant $2\pi$ on the right-hand side of \eqref{conc2d} corresponds to the half of the total Gaussian curvature of $\mathbb S^2$, a feature that will appear again.

That case \eqref{conc2d} actually occurs can be easily seen by considering the function
$$u(x)=\log\frac{2}{1+|x|^2},$$
and then defining $u_k(x)=u(kx)+\log k$. Each $u_k$ solves \eqref{eqBM} in $\Omega=\R^2$ with $K_k\equiv 1$ and $\|e^{2u_k}\|_{L^1(\R^2)}=4\pi$, so that $e^{2u_k} \stackrel{*}{\rightharpoonup} 4\pi \delta_0$.
\medskip

In higher even dimension $2m$ one can replace equation \eqref{eqBM} with
\begin{equation}\label{eq2m}
(-\Delta)^m u_k = Q_k e^{2mu_k}\quad \text{in }\Omega\subset \R^{2m},
\end{equation}
having the geometric interpretation that $Q_k$ is the $Q$-\emph{curvature} of the conformal metric $e^{2u_k}|dx|^2$ on $\Omega$. In spite of similar scaling properties, as discovered by Adimurthi, Robert and Struwe \cite{ARS}, equation \eqref{eq2m} exhibits a richer blow-up behaviour than \eqref{eqBM} when $2m=4$. In particular blow-up is possible not only on isolated points, but also on hyperplanes or, in general, on zero sets of non-positive biharmonic functions. This was later generalized to arbitrary even dimension $2m\ge 4$ in \cite{MarCC}. For a finite set $S_1\subset \Omega$ define
\begin{equation}\label{defK}
\mathcal{K}(\Omega, S_1):=\{\varphi\in C^\infty(\Omega \setminus S_1):\varphi\le 0,\,\varphi\not\equiv 0,\, \Delta^m \varphi\equiv 0\},
\end{equation}
and for a function $\varphi \in \mathcal{K}(\Omega,S_1)$ set
\begin{equation}\label{defS0}
S_\varphi:=\{x\in\Omega\setminus S_1: \varphi(x)=0\}.
\end{equation}

\begin{thm}[\cite{ARS,MarCC}]\label{trmARSM}
Let $(u_k)$ be a sequence of solutions to \eqref{eq2m} for some $m\ge 1$ under the bounds
\begin{equation}\label{bounds2m}
\|Q_k\|_{L^\infty(\Omega)}\le C,\quad \int_{\Omega}e^{2mu_k}dx\le C.
\end{equation}
Then the set
$$S_1:=\left\{x\in \Omega: \lim_{r\downarrow 0}\limsup_{k\to\infty}\int_{B_r(x)}|Q_k|e^{2mu_k}\,dy\ge \frac{\Lambda_1}{2} \right\},\quad\text{where}\quad \Lambda_1:=(2m-1)!\,\vol(\mathbb S^{2m}),$$
is finite (possibly empty) and, up to a subsequence, either
\begin{itemize}
\item[$(i)$] $(u_k)$ is bounded in $C^{2m-1,\alpha}_{\loc}(\Omega\setminus S_1)$, or
\item[$(ii)$] there exists a function $\varphi\in\mathcal{K}(\Omega, S_1)$ and a sequence $\beta_k\to+\infty$ as $k\to+\infty$ such that
$$\frac{u_k}{\beta_k}\to \varphi \text{ locally uniformly in }\Omega\setminus S_1.$$
In particular, $u_k\to -\infty$ locally uniformly in $\Omega\setminus (S_\varphi\cup S_1)$.
\end{itemize}
\end{thm}

In fact, in case (\emph{ii}) of Theorem \ref{trmARSM}, one can prescribe the blow-up set $S_\varphi$, in the sense that given any $\varphi\in \mathcal{K}(\Omega,\emptyset)$, one can construct a sequence $(u_k)$ solving \eqref{eq2m} and \eqref{bounds2m} with $u_k\to +\infty$ on $S_\varphi$, as shown in \cite{HIM}. Moreover in the radial case of dimension $6$, it was also shown in \cite{HM} that the blow-up set $S_1=\{0\}$ and $S_\varphi=\{x:|x|=1\}$ can coexist. See also \cite{DR, Mal, Ndi} for the case of a closed manifold of even dimension $4$ and higher.

\medskip

The problem of prescribing $Q$-curvature is not confined to even dimensions, but a crucial difficulty that arises when studying a problem as \eqref{eq2m} in any odd dimension $n$ is that one has to deal with the fractional Laplacian operator $(-\Delta)^\frac{n}{2}$, which is non-local. This was done in dimension one in the cases of $\mathbb S^1$ and of the real line by Da Lio, Martinazzi and Rivi\`ere \cite{DM,DMR}. In particular, the following compactness result is proven:

\begin{thm}[\cite{DM}]\label{trmDM} Let $(u_k)\subset L_\frac12(\R)$ be a sequence of solutions to
\begin{equation}\label{equk}
(-\Delta)^\frac{1}{2}u_k=K_ke^{u_k}\quad \text{in } \R{}
\end{equation}
and assume that
\begin{equation}\label{boundKL}
\|K_k\|_{L^\infty}\le C,\quad \int_{\R}e^{u_k}dx \le C.
\end{equation}
Up to a subsequence assume that $K_k\stackrel{*}{\rightharpoonup}  K_{\infty}$ in $L^{\infty}(\R)$.
Then there exists a finite (possibly empty) set $B:=\{x_1,\ldots, x_N\}\subset \R$ such that, up to extracting a further subsequence, one of the following alternatives holds:
\begin{itemize}
\item[(i)] $u_k\to u_\infty$ in $W^{1,p}_{\loc}(\R\setminus B)$ for $p<\infty$, where
 \begin{equation} \label{quantDM}
K_ke^{u_k} \stackrel{*}\rightharpoonup K_{\infty} e^{u_\infty}+ \sum_{i=1}^{N}\pi  \delta_{x_i}.
\end{equation}
\item[(ii)] $u_k\to -\infty$ locally uniformly in $\R\setminus B$ and
 \begin{equation} \label{quantDM2}
K_ke^{u_k} \stackrel{*}\rightharpoonup \sum_{j=1}^N \alpha_j \delta_{x_j},\quad \text{for some }\alpha_1,\dots,\alpha_N\ge \pi.
\end{equation}
\end{itemize}
\end{thm}

The geometric interpretation of \eqref{equk} is not in terms of intrinsic curvatures, but rather of the curvature of $\Phi|_{\mathbb S^1}: S^1\to \mathbb{C}$, where $\Phi: D\to \mathbb{C}$ is a conformal immersion of the unit disc of the complex plane. The constant $\pi$ appearing in \eqref{quantDM} and \eqref{quantDM2} corresponds to half the total curvature of $\mathbb  S^1$ and is the $1$-dimensional analog of the constant $\frac{\Lambda_1}{2}$ appearing in Theorem \ref{trmARSM}. It arises as consequence of a pinching phenomenon, as already described in \cite{DMR}. Notice that Theorem \ref{trmDM} is more general than Theorem \ref{thbm} as no assumption on the sign of $K_k$ is made and, in fact, to have $N>0$ in case \emph{(i)} it is necessary that $K_k$ changes sign near the blow-up points.

\medskip

In this paper we shall focus on the $3$-dimensional case. In particular, on the $3$-dimensional analog of \eqref{eq2m},
suitably defined. Instead of the geometric interpretation of \eqref{equk} in terms of conformal immersions, we will consider the function $u$ to be the trace of a function defined in all of the half-space $\R^4_+$. This leads to a different geometric interpretation in terms of conformal geometry and fractional $Q$-curvature, which is the natural setting to understand \eqref{eq2m} and a curved generalization on it.

More precisely, let us denote any point $X\in\R^4$ by $X=(x,y)=(x_1,x_2,x_3,y)$ and set $\mathbb R^4_+=\{(x,y)\,:\,y>0\}$. We will identify $\mathbb R^3=\{(x,y)\,:\,y=0\}=\partial \R^4_+$. In the following, $\Delta$ will denote the Laplacian in $\R^4_+$ and $\Delta_x$ the Laplacian in $\R^3$.

Assume that $U\in C^\infty(\overline{ \R^4_+})$ is a solution to the problem
\begin{equation}\label{extension0}
\left\{
\begin{split}
&\Delta^2 U=0\quad\text{in }\R^4_+,\\
&\partial_y U=0\quad\text{on }\R^3.
\end{split}
\right.
\end{equation}
Let $u$ be its Dirichlet data $u=U|_{\R^3}$. If $U\in W^{2,2}(\mathbb R^4_+)$, then, as we shall see, $U$ is characterized by the Poisson representation formula
\begin{equation}\label{kernel-Dirichlet-introduction}
U(x,y)=\frac{4}{\pi^2} \int_{\mathbb R^3}\frac{y^3}{(y^2+|x-\tilde x|^2)^3} u(\tilde x)\,d\tilde x.
\end{equation}
Define the operator on $\mathbb R^3$ by
\begin{equation}\label{ope_def}
\mathcal L_{\frac{3}{2}}U:=\frac{1}{2}\lim_{y\rightarrow 0}\partial_y \D U,
\end{equation}
Then it is known (see Proposition \ref{prop:relation1} in Section \ref{section:kernel}) that
\begin{equation*}\label{relation}
\mathcal L_{\frac{3}{2}}U=(-\Delta_x)^{\frac{3}{2}} u,
\end{equation*}
where the $\frac{3}{2}$-fractional Laplacian is defined as the operator with Fourier symbol $|\xi|^{3}$.
Note that $\mathcal L_{\frac{3}{2}}$ can also be defined in a distributional sense. Indeed, given $u\in L_{loc}^1(\mathbb R^3)$, we say that $U\in W^{2,2}(\mathbb R^4_+)$ satisfying $\partial_y U=0$ on $\mathbb R^3$ and $\D^2 U=0$ in $\R^4_+$ is a weak solution to
$\mathcal L_{\frac{3}{2}} U=w, $
 if
\begin{equation*}
0=\int_{\mathbb R^4_+} \Delta U\Delta \psi\,dxdy-2\int_{\mathbb R^3} w\psi \,dx
\end{equation*}
for every test function $\psi\in\mathcal C^{\infty}(\mathbb R^4_+)$ with compact support in $\overline{\mathbb R^4_+}$
and satisfying $\partial_y \psi=0$ on $y=0$.

However, to include a larger class of solutions, we will need to admit functions with polynomial growth at infinity and thus, not in the energy class $W^{2,2}$. Then uniqueness is lost in general, in the sense that several solutions $U$ to \eqref{extension0} can have the same Dirichlet datum $U|_{\R^3}$. Of course this is not the case if we restrict to solutions as in \eqref{kernel-Dirichlet-introduction}.

\begin{defn}\label{defLs}
We say that a solution $U$ to \eqref{extension0} is \emph{representable} if it coincides with its Poisson formula representation \eqref{kernel-Dirichlet-introduction} and $\mathcal L_{\frac{3}{2}}U$ is well defined. In particular, a representable solution has boundary value $U|_{y=0}$ belonging to the space $L_6(\R^n)$ where
$$L_s(\R^n):=\left\{ u\in L^1_{loc}(\R^n):\int_{\R^n}\frac{|u(x)|}{1+|x|^s}\,dx<\infty \right\}.$$

\end{defn}

We will give sufficient conditions for representability in Proposition \ref{prop:kernel3}.\\

Let $\Sigma_0$ be a smoothly bounded domain in $\mathbb R^3$. We would like to study the (localized) nonlinear equation
 \begin{equation}\label{equation-extension-introduction}\left\{ \begin{split}
  \D^ 2 U=0 \quad \text{ in }\R^4_+,\\
   \partial_y U=0 \quad \text{ in }\R^3,\\
    \mathcal L_{\frac{3}{2}}U=Q(x)e^{3u} \quad\text{on }  \Sigma_0,
  \end{split}\right.
 \end{equation}
 where, as above, $u:=U|_{y=0}$. Of course, this is equivalent to the equation
 \begin{equation}\label{Liouville}
(-\Delta_x)^{\frac{3}{2}}u=Q(x)e^{3u} \quad\text{in }\Sigma_0.
\end{equation}

The interpretation of \eqref{Liouville} in conformal geometry will be further explained in Section \ref{section:kernel} (see \cite{Chang-Qing:zeta1,Chang-Qing:zeta2,Case:boundary-operators} and the survey \cite{Chang:survey}). Indeed, on the boundary $M^3$ of a 4-dimensional manifold it is possible to define a third order curvature, the $T$-\emph{curvature} , in relation to a 4-dimensional Gauss-Bonnet formula for manifolds with boundary \cite{Chang-Qing-Yang}. This $T$-curvature satisfies the conformal property
 \begin{equation}\label{equation-T-introduction}
 P^{g} u +T^{g}=T^{\tilde g}e^{3u},
 \end{equation}
under the conformal change of metric $\tilde g=e^{2u}g$, where $P$ is a third order boundary operator, corresponding to the (fourth-order) Paneitz operator on the four-dimensional ambient manifold. In the flat setting, \eqref{equation-T-introduction} reduces to \eqref{Liouville} for a conformal metric $\tilde g=e^{2u}|dx|^2$ on $\mathbb R^3$.

In some particular cases, this operator $P$ can be understood as the limit $\gamma\to 3/2$ of the conformal fractional Laplacian $P_\gamma$ (see, for instance, \cite{Chang-Gonzalez,Case-Chang}, the survey \cite{Gonzalez:survey} and the references therein for the necessary background). $P_\gamma$
is a non-local operator with principal symbol the same as $(-\Delta_M)^\gamma$ and, in the flat case, $P_{\gamma}=(-\Delta_x)^{\gamma}$.\\

Our main theorem studies concentration phenomena for the Liouville equation \eqref{Liouville}
in dimension 3. 
In analogy with Theorem \ref{trmARSM} we will see that solutions can blow up on isolated points and also on the zero set of certain biharmonic functions. More precisely, let us set
\begin{equation}\label{K}
\begin{split}
\K(\Sigma_0):=\big\{& H\in C^\infty(\R^4_+\cup \Sigma_0): \D^2H=0,\,H\le 0, \text{ in }\mathbb R^4_+,\\
&\qquad\qquad\qquad\quad   H\not\equiv 0,\, \partial_yH=\mathcal{L}_\frac32H=0\text{ on }\Sigma_0\big\}.
\end{split}
\end{equation}

\begin{thm} \label{thm1}
Let $(U_k)\subset C^0(\overline{\R^4_+})$ be a sequence of representable functions satisfying
 \begin{equation}\label{eq-1ext}\left\{ \begin{split}
  \D^ 2 U_k=0 \quad \text{ in }\R^4_+,\\
   \partial_y U_k=0 \quad \text{ in }\R^3,\\
    \mathcal L_{\frac{3}{2}}U_k=Q_ke^{3u_k} \quad\text{on }  \Sigma_0,
  \end{split}\right.
 \end{equation}
where $u_k:=U_k|_{y=0}$ and $Q_k\in C^0(\overline{\Sigma_0})$ is uniformly bounded in $L^\infty(\Sigma_0)$.
We assume that
\begin{align*}\label{cond-vol}
 \int_{\Sigma_0}e^{3u_k}\,dx\leq C.
\end{align*}
and
\begin{equation}\label{extra-assumption}
\int_{\R^3}\frac{u_k^+(x)}{1+|x|^6}\,dx\leq C.
\end{equation}
Set
$$S_1:=\left\{\bar x\in\Sigma_0:\lim_{\ve\to 0^+}\liminf_{k\to\infty}\int_{B_\ve(\bar x)}|Q_k|e^{3u_k}\,dx\geq \frac{\Lambda_1}{2} \right\},\quad \Lambda_1=2|\mathbb S^3|=4\pi^2.$$
Then $S_1$ is a finite set and, up to a subsequence, one of the following is true:
\begin{itemize}
 \item [(i)] $U_k\to U_\infty$ in $C^{2,\alpha}_{loc}(\R^4\cup (\Sigma_0\setminus S_1))$ for any $\alpha\in [0,1)$,


 \item [(ii)] There exists $\Phi\in \K(\Sigma_0)$ and numbers $\beta_k\to\infty$ such that
 $$\frac{U_k}{\beta_k}\to\Phi\quad\text{in }C^{2,\alpha}_{\loc}((\R^4_+\cup \Sigma_0)\setminus S),\quad S=S_\Phi \cup S_1,$$
 where
$S_\Phi:=\{x\in \Sigma_0:\Phi(x)=0\}$.
Moreover $S_\Phi$ has dimension at most $2$.
\end{itemize}
\end{thm}


\begin{rem}
Since the operator $(-\Delta_x)^{\frac{3}{2}}$ is  non-local, the extra assumption \eqref{extra-assumption} is needed in order to control the behavior of $u_k$ outside of $\Sigma_0$ (where the equation lives), but it does not prevent concentration happening. Indeed,  take $u$ to be the model concentration solution to \eqref{Liouville} in $\R^3$ with finite volume condition $\int_{\R^3}e^{3u}\,dx<\infty$. This is, of the form  $u(x)\approx -\log|x|-c|x|^2$, $c\geq 0$. Then $u_k(x)=u(kx)+\log k$ also satisfies the same equation with same volume. Now we define $U_k$ using the extension \eqref{extension0}. Clearly $U_k$ satisfies the assumptions of Theorem \ref{thm1}.
\end{rem}

\begin{rem}
Let us make a comment on localization. Note that if $h$ is a harmonic function in $\R^3$ then $H(x,y)=h(x)y^2\in \mathcal{H}$, where we have defined
\begin{equation}\label{H}
\mathcal{H}:=\big\{H\in C^\infty(\overline{\mathbb R^4_+}): \D^2H=0 \text{ in }\mathbb R^4_+,\, H=\partial_yH=0\text{ on }\mathbb R^3,\, \mathcal{L}_\frac32H=0\text{ on }\Sigma_0\big\}.
\end{equation}
One could localize the first two equations in \eqref{equation-extension-introduction} to a subset $\Sigma=\Omega\cap \{y\geq 0\}$,  where $\Omega$ is a smoothly bounded open subset of $\R^4$ intersecting $\R^3\times\{0\}$. However, by working on $\mathbb R^4_+$ with globally defined representable solutions  we avoid the presence of the kernel \eqref{K} and the extension function $U$ is unique for each $u$. In other words, if for a solution $U$ of \eqref{equation-extension-introduction} one removes the  representability assumption, we could construct a sequence $U_k=U\pm ky^2$ (or $U_k=U+H_k$ with $H_k\in\K$) still solving \eqref{equation-extension-introduction}. Here the sequence $(U_k)$ is unbounded, however, $(u_k)$ is bounded.
This is in agreement, again, with the general fact that the fractional Laplacian is a non-local operator.
\end{rem}



Several ideas in the proof of Theorem \ref{thm1} rely on the paper \cite{ARS} on concentration phenomena for a fourth-order Liouville's equation in dimension four. Both are inspired from the two-dimensional case \cite{Brezis-Merle,Li-Shafrir}, where the main step is to prove a Brezis-Merle estimate. This is done in Lemma \ref{lemmaBM}.\\

We next show that case \emph{(ii)} of Theorem \ref{thm1} is non-trivial:

\begin{prop}\label{prop:existence}
Let  $\Phi \in \mathcal{K}(\Sigma_0)$ solve $\partial_y \Phi(x,0)=0$ for every $x
\in \R^3$, assume further that it is representable and set
$$S_\Phi:=\{x\in\Sigma_0:\Phi(x)=0\}.$$
Then, given a sequence $(Q_k)\subset L^\infty(\Sigma_0)$ with $\|Q_k\|_{L^\infty(\Sigma_0)}\leq C$, there exists a sequence of solutions $(U_k)\subset C^0(\overline{\R^4_+})$ to \eqref{eq-1ext} such that $U_k\to\infty$ on $S_\Phi$ and case  $(ii)$ of Theorem \ref{thm1} holds with $S_1=\emptyset$.
\end{prop}

It remains open to study the asymptotic behavior near concentration points. We hope to return to this problem elsewhere. We also expect that results similar to Theorem \ref{thm1} and Proposition \ref{prop:existence} hold in higher dimension, although some technical difficulties might appear.\\

The crucial difficulty in the proofs in this paper is the fact that the operator $(-\Delta)^{\frac{3}{2}}$ is non-local, and some of the usual arguments for elliptic local problems cannot be used. We are able to deal with this issue by passing to a local equation in the extension. While this is a common procedure to handle the fractional Laplacian $(-\Delta_x)^\gamma$ in $\mathbb R^n$, for powers $\gamma\in(0,1)$, here we present the corresponding scheme for higher powers of the Laplacian $\gamma\in(0,\frac{n}{2}]$. In particular, we obtain explicit formulas for solutions of the poly-harmonic equation in the upper half-space such as \eqref{extension0}, in terms of its Dirichlet data and its Neumann-type data \eqref{ope_def}. These formulas are valid in any odd dimension $n$.

\begin{prop}\label{prop:kernel}
Let $n=2m+1$ be an odd integer. For $U\in W^{m+1,2}(\R^{n+1}_+)$ we consider the extension problem
\begin{equation}\label{problem-extension}\left\{
\begin{split}
&\Delta^{m+1}  U=0\quad\text{in }\mathbb R^{n+1}_+,\\
&\partial_y^{2j+1} U(\cdot,0)=0\quad\text{on }\mathbb R^n,\quad \text{for every}\quad 2j+1<m+1, \quad\text{and}\\
&\partial^{2j}_y U(\cdot, 0)  = (\Delta_x)^{j} U(\cdot, 0)\prod_{l=1}^j
\frac{1}{2m+1-4(l-1)}\quad\text{on }\mathbb R^n,\quad \text{for every}\quad 2j<m+1,
\end{split}\right.
\end{equation}
with a Dirichlet condition
\begin{equation}\label{Dirichlet-condition}
U=f\quad\text{on }\mathbb R^{n}.
\end{equation}
Then a solution can be written by the Poisson formula
\begin{equation}\label{kernel-Dirichlet}
U(x,y)=\int_{\mathbb R^n} {\mathcal K}_{\frac{n}{2}}(x-\tilde x,y) f(\tilde x)\,d\tilde x,
\end{equation}
where
\begin{equation}\label{kernel-D}
\mathcal K_{\frac{n}{2}}(x,y):=
\kappa_n\frac{y^{n}}{(y^2+|x|^2)^{n}},\quad \kappa_n=\frac{\Gamma\left(n\right)}{\Gamma{(\frac{n}{2})}\pi^{\frac{n}{2}}}.
\end{equation}
If one considers the extension problem \eqref{problem-extension} with a Neumann-type condition
\begin{equation}\label{Neumann-condition}
-\frac{\Gamma(m+\frac{1}{2})}{ m!\sqrt \pi}\lim_{y\rightarrow 0}\partial_y \Delta^{m} U=w\quad\text{on }\mathbb R^n,
\end{equation}
instead of \eqref{Dirichlet-condition}
then a solution can be written as
\begin{equation*}
U(x,y)=\int_{\mathbb R^n} \tilde{\mathcal K}_{\frac{n}{2}}(x-\tilde x,y) w(\tilde x)\,d\tilde x,
\end{equation*}
where
\begin{equation}\label{kernel}
\tilde{\mathcal K}_{\frac{n}{2}}(x,y)
=\tilde\kappa_n\log\frac{1}{y^2+|x|^2},\quad  \tilde\kappa_n=
\frac{1}{2^n\Gamma{(\frac{n}{2})}\pi^{\frac{n}{2}}}.
\end{equation}
\end{prop}

This paper is structured as follows: in Section \ref{section:representation} we prove the two representation formulas that are needed for the proof of the main theorem, that is contained in Section \ref{section:proof}. Then, in Section \ref{section:existence} we consider the existence result from Proposition \ref{prop:existence}. Finally, Section \ref{section:kernel} is of independent interest, and it contains the proof of the representation formulas from Proposition \ref{prop:kernel} using techniques from conformal geometry. Moreover, we explain here the geometric content of \eqref{Liouville}.

\section{Representation formulas (in dimension $n=3$)}\label{section:representation}

General representation formulas for poly-harmonic functions in the upper half-space 
were given in Proposition \ref{prop:kernel}. However, these are proven using Fourier transform arguments and are well suited for energy solutions $U\in W^{m,2}(\R^{n+1}_+)$. In the following, we concentrate on biharmonic functions in dimension $n+1=4$, and move outside the energy class. More precisely, we prove the following uniqueness result (note that this is not the sharpest possible statement, but it is enough for our purposes):

\begin{prop}\label{prop:kernel3}
Given $u\in L_6(\R^3)$ (compare to Definition \ref{defLs}), the function
\begin{equation}\label{kernel-Dirichlet-representation}
U(x,y)=\int_{\mathbb R^3} {\mathcal K}_{\frac{3}{2}}(x-\tilde x,y) u(\tilde x)\,d\tilde x,
\end{equation}
with
\begin{equation*}
{\mathcal K}_{\frac{3}{2}}(x,y)=\frac{4}{\pi^2}\frac{y^3}{(y^2+|x|^2)^3}
\end{equation*}
solves the extension problem
\begin{equation}\label{problem-extension-3}\left\{
\begin{split}
&\Delta^{2}  U=0\quad\text{in }\mathbb R^{4}_+,\\
&U(\cdot,0)=u,\quad \partial_y U(\cdot,0)=0\quad\text{on }\mathbb \partial\R^4_+\simeq \R^3,
\end{split}\right.
\end{equation}
in the sense that  for every  $\varphi\in C^\infty_c(\overline{\R^4_+})$ with $\vp=\partial_y\vp=0$ on $\R^3$ we have
\begin{equation}\label{weakext}\int_{\R^4_+}U(X)\Delta^2\varphi(X)\,
 dX=-\int_{\R^3}u(x)\partial_y\Delta\varphi(x,0)\,dx.
\end{equation}
Moreover $U$ is the unique weak solution to \eqref{problem-extension-3} (i.e. \eqref{weakext}) among the functions $U\in L^1_{\loc}(\overline{\R^4_+})$ satisfying the bound
$$|U(X)|+y|\partial_y U(X)|+y^2|\D U(X)|\leq C(1+|X|^N),\quad \lim_{y\to\infty}\frac{U(x,y)}{y^2}=0\quad\forall\,x\in\R^{3},$$ for some $N\geq 1$.
\end{prop}

\begin{proof}
That $U$ as defined in \eqref{kernel-Dirichlet-representation} solves \eqref{problem-extension-3} in the strong sense when $u\in C^\infty_c(\R^3)$ follows from \eqref{kernel-Dirichlet} for $n=3$.
The general case is proven by approximation. The uniqueness part follows at once from the following Proposition.
\end{proof}

\begin{prop}Let $U\in L^1_{\loc}(\overline{\R^4_+})$ be a solution to $$\D^2U=0\quad\text{in }\R^4_+,\quad U(\cdot,0)=\partial_y U(\cdot,0)=0\quad\text{on }\R^{3},$$ in the sense that
$$\int_{\R^4_+}U(X)\Delta^2\varphi(X)\,
 dX=0\quad \text{for every }\vp\in \mathcal{S},$$ where $$\mathcal{S}:=\{\vp\in C_c^\infty(\overline{\R^4_+}):\vp=\partial_y\vp=0\text{ on }\R^3\},$$
and $X=(x,y)\in \R^3\times [0,\infty)$.
Then $U$ has a bi-harmonic extension on $\R^4$. Moreover, if $U$ satisfies
$$|U(X)|+y|\partial_y U(X)|+y^2|\D U(X)|\leq C(1+|X|^N),\quad \lim_{y\to\infty}\frac{U(x,y)}{y^2}=0\quad\forall\,z\in\R^{3},$$ for some $N\geq 1$, then $U\equiv0$.
\end{prop}
\begin{proof} We shall follow \cite[Lemma 2.3]{reflection}. We define the distribution
$$\langle \tilde U,\vp\rangle:=\int_{\R^4_+}U(X)\big\{\vp(X)-5\vp(X^*)+6y(\partial_y\vp)(X^*)-y^2(\D\vp)(X^*) \big\}\,dX,\quad \vp\in C_c^\infty (\R^4),$$ where for $X=(x,y)\in\R^{3}\times\R$ we have set $X^*=(x,-y)$. We claim that $\tilde U$ is the unique bi-harmonic extension of $U$ in $\R^4$. Uniqueness follows immediately since bi-harmonic distributions are analytic.

First we show that $\tilde U$ is indeed an extension of $U$.  For every $\vp\in C_c^\infty(\R^4_+)$ we have $$\langle \tilde U,\vp\rangle=\int_{\R^4_+}U\vp \,dX,$$ and hence $U=\tilde U$ on $\R^4_+$.  Next, to show that $\tilde U$ is bi-harmonic in the weak sense, we compute for $\vp\in C_c^{\infty}(\R^4)$
\begin{align*}\langle \tilde U,\D^2\vp\rangle&=\int_{\R^4_+}U(X) \big\{\D^2\vp(X)-5\D^2\vp(X^*)+6y(\partial_y\D^2\vp)(X^*)-y^2\D^3\vp(X^*) \big\}\,dX\\ &=\int_{\R^4_+}U(X)\D^2\Phi(X)\,dX,\end{align*}
where
$$\Phi(X):=\vp(X)-\vp(X^*)-y^2\D\vp(X^*)-2y(\partial_y\vp)(X^*).$$
Notice that $\Phi|_{\R^4_+}\in\mathcal{S}$.
Therefore
$$\langle \tilde U,\D^2\vp\rangle=0\quad\text{for every }\vp\in C_c^\infty(\R^4),$$ in other words, $\tilde U$ is bi-harmonic in $\R^4$.

Now we prove the second part of the proposition.  From the definition one can show that $$\tilde U(X)=-U(X^*)-2y(\partial_y U)(X^*)-y^2(\D U)(X^*)\quad\text{for }X\in\R^4_-.$$ It follows from the growth assumptions on $U$ and its derivatives that $|\tilde U(X)|\leq C(1+|X|^N)$, and therefore, by Liouville theorem we obtain that $\tilde U$ is a polynomial, that is, we can write  $$\tilde U(X)=\sum_{i=0}^mP_i(x)y^i,$$ for some polynomials $P_i$ in $\R^{3}$. The boundary conditions $U=\partial_y U=0$ on $\R^{3}$ imply that $P_0\equiv0 \equiv P_1$. Finally, the assumption $U(x,y)=o(y^2)$ implies that $P_i\equiv 0$ for $i\geq 2$. This completes the proof. \end{proof}

We note here that formulas for higher order extension problems in the ball have been considered in \cite{Abatangelo-Jarohs-Saldana:1,Abatangelo-Jarohs-Saldana:2}.\\

Next, we give an expression for a solution to \eqref{problem-extension-3} in terms of its third-order Neumann  data $\mathcal L_{\frac{3}{2}} U$ given in a subset $\Sigma_0\subset\mathbb R^3$. Recall, from Proposition \ref{prop:kernel} and the subsequent discussion, that
  $$\frac{1}{2\pi^2}(-\D_x)^\frac 32\log\frac{1}{|x|}=\delta_0\quad\text{in }\R^3.$$
Then it is natural to define, for $w\in L^1(\Sigma_0)$,
\begin{equation}\label{def-v}
V(x,y)=\frac{1}{2\pi^2}\int_{\Sigma_0}\log\left(\frac{1}{|(x,y)-(\tilde x,0)|}\right)w(\tilde x)\,d\tilde x,\quad (x,y)\in \R^3\times\R.
\end{equation}
It is easy to see that $V\in C^\infty(\R^4\setminus \bar\Sigma_0)$ and $V$ is well defined for almost every $(x,0)\in\Sigma_0$.

\begin{defn} An open set $\Omega\subset\R^4$ shall be called \emph{admissible} if it is bounded and $\Omega\cap \R^3\Subset\Sigma_0$.
\end{defn}

\begin{lem}\label{vk}
Given $w\in L^1(\Sigma_0)\cap C^0(\Sigma_0)$  and $V$ given by \eqref{def-v} we have
\begin{itemize}
\item[i)] $\D^2 V=0$ in $\R^4\setminus\bar \Sigma_0$.
\item[ii)]  $\displaystyle \lim_{y\to0}\frac{\partial}{\partial y}V(x,y)=0$ for every $x\in \R^3$.
\item[iii)] $\mathcal L_{\frac{3}{2}} V=w$ on $\Sigma_0$.
\item[iv)] $\displaystyle\int_{\Omega}|V(x,y)|\,dxdy\leq  C(\Omega,\Sigma_0)\|w\|_{L^1(\Sigma_0)}$ for every bounded domain $\Omega\subset\R^4$.
\item[v)] $V\in C^{2,\alpha}(\Omega)$ for every admissible set $\Omega$.
\end{itemize}
\end{lem}
\begin{proof}
Most of the statements in this lemma are contained in the proof Proposition \ref{prop:kernel}. However, let us give a direct proof.

Since $\log |\cdot|$ is a fundamental solution for $\D^2$ in $\R^4$, we have $i)$.
Next, differentiating under the integral sign, from \eqref{def-v}, one has
$$\frac{\partial}{\partial y}V(x,y)=-\frac{1}{2\pi^2}\int_{\Sigma_0}\frac{y}{|x-\tilde x|^2+y^2}w(\tilde x)\,d\tilde x.$$
Then $ii)$ follows by dominated convergence theorem (with dominating function $\frac{w(\tilde x)}{|x-\tilde x|^2}$).

Again differentiating under the integral sign we get
\begin{equation}\label{Lap_V}
\D V(x,y)=-\frac{1}{\pi^2}\int_{\Sigma_0}\frac{1}{|x-\tilde x|^2+y^2}w(\tilde x)\,d\tilde x.
\end{equation}
This gives
$$\frac{\partial}{\partial y}\D V(x,y)=\frac{2}{\pi^2}\int_{\Sigma_0}\frac{y}{(|x-\tilde x|^2+y^2)^2}w(\tilde x)\,d\tilde x,\quad y\neq 0.$$
Now we fix $x\in\Sigma_0$. Since $w\in C^0(\Sigma_0)$, given $\ve>0$ we have  $w(\tilde x)=w(x)+o_\ve$ on $B_\epsilon(x)$, with $o_\ve\to0$ as $\ve\to0$. Therefore, since $\lim_{y\to0}\int_{\Sigma_0\setminus B_\epsilon(x)}\frac{y}{|x-\tilde x|^2+y^2}w(\tilde x)\,d\tilde x=0$,
\begin{align*}
\lim_{y\to0}\frac{\partial}{\partial y}\D V(x,y)&=\frac{2}{\pi^2}\left(w(x)+o_\ve\right)\lim_{y\to0}\int_{B_\ve(x)\subset\R^3}\frac{y}{(|x-\tilde x|^2+y^2)^2}\,d\tilde x\\
&=\frac{2}{\pi^2}\left(w(x)+o_\ve\right)\int_{\R^3}\frac{1}{(|z|^2+1)^2}\,dz\\
&\xrightarrow{\ve\to0}2w(x),
\end{align*}
and this yields $iii)$. Note that in the second last to last line we have used that
\begin{equation*}
\begin{split}
\int_{\R^3}\frac{1}{(1+|z|^2)^2}\,dz&=|\mathbb S^2|\int_0^\infty \frac{r^2}{(1+r^2)^2}\,dr=4\pi \int_0^\infty r\left(\frac{-1}{2(1+r^2)}\right)'\,dr\\
&=2\pi\int_0^\infty\frac{1}{1+r^2}\,dr=\pi^2.
\end{split}\end{equation*}
But, using Fubini's theorem,
$$\int_{\Omega}|V(x,y)|\,dxdy\leq \int_{\Sigma_0}|w(\tilde x)| \int_{\Omega}\big|\log|(x,y)-(\tilde x,0)|\big|\,dxdyd\tilde x\leq C(\Omega,\Sigma_0)\|w\|_{L^1(\Sigma_0)},$$
as desired.

Finally, regularity follows by standard arguments directly from \eqref{Lap_V}.
  \end{proof}




\section{Proof of Theorem \ref{thm1}}\label{section:proof}

Let us assume that $Q_ke^{3u_k}\in L^1(\Sigma_0)$. In the light of \eqref{def-v}, we set
\begin{align}\label{def-vk}
V_k(x,y)=\frac{1}{2\pi^2}\int_{\Sigma_0}\log\left(\frac{1}{|(x,y)-(\tilde x,0)|}\right)Q_k(\tilde x)e^{3u_k(\tilde x)}\,d\tilde x,\quad (x,y)\in \R^3\times\R,
\end{align}

We will use the following notation:
\begin{equation*}
\begin{split}
\B_R(X_0)&=\{(x,y)\in\R^{4}\,:\, |(x,y)-(x_0,y_0)|<R\},\quad X_0=(x_0,y_0)\in \R^4,\\
B_R(x_0)&=\{x\in \R^3: |x-x_0|<R\},\quad x_0\in \R^3.
\end{split}
\end{equation*}

The following lemma can be seen as a Brezis-Merle-type estimate.

\begin{lem}\label{lemmaBM} For every $K\Subset (\R^4\setminus S_1)$ there exists $p>1$ such that
\begin{align}
\int_{K\cap \Sigma_0}e^{3pV_k(x,0)}\,dx&\le C(K)\label{BM2},
\end{align}
uniformly with respect to $k$.
\end{lem}

\begin{proof} For every $X\in K$ we can find $R_X\le \tfrac{1}{4}$ such that
$$\liminf_{k\to\infty}\int_{\mathbb B_{2R_X}(X)\cap \Sigma_0} {|Q_k|} e^{3u_k}\,dx <\frac{\Lambda_1}{2}.$$
By compactness we can extract a finite covering, i.e. points $X_1,\dots, X_M$ such that for $R_j:=R_{X_j}$,
$$K\subset\bigcup_{j=1}^M \mathbb B_{R_j}(X_j)$$
and up to extracting a subsequence we can assume that
$$\limsup_{k\to\infty}\int_{\mathbb B_{2R_j}(X_j)\cap \Sigma_0} |Q_k|e^{3u_k}\,dx <\frac{\Lambda_1}{2}(1-\delta),$$
for some $\delta =\delta(K)>0$.
Fix $j\in\{1,...,M\}$. For $X\in \mathbb B_{R_j}(X_j)$ we bound
\begin{align*}
V_k(X) &=\frac{1}{2\pi^2}\int_{\mathbb B_{2R_j}(X_j)\cap \Sigma_0}\log\left(\frac{1}{|X-(\tilde x,0)|}\right)Q_k(\tilde x)e^{3u_k(\tilde x)}\,d\tilde x\\
&\quad +\frac{1}{2\pi^2}\int_{\Sigma_0\setminus \mathbb B_{2R_j}(X_j)}\log\left(\frac{1}{|X-(\tilde x,0)|}\right)Q_k(\tilde x)e^{3u_k(\tilde x)}\,d\tilde x\\
&{=} (I)_j+(II)_j.
\end{align*}
Observe that
$$|(II)_j|\le C|\log R_j| \|Q_ke^{3u_k}\|_{L^1(\Sigma_0)}\le C_j(K).$$
Assuming that
$$\alpha_k:=\|Q_ke^{3u_k}\|_{L^1(\mathbb B_{2R_j}(X_j)\cap \Sigma_0)}>0$$
(otherwise $(I)_j=0$), we can use Jensen's inequality with
$$d\mu_k(\tilde x)= \frac{|Q_k(\tilde x)|e^{3u_k(\tilde x)}}{\alpha_k}\, d\tilde x,$$
and using that for $k$ large enough
$$\frac{\alpha_k}{2\pi^2}\le (1-\delta),$$
we get for $p<\frac{1}{1-\delta}$
\begin{align*}
\int_{\mathbb B_{R_j}(X_j)\cap \Sigma_0} &e^{3pV_k(x)}\,dx\\
&\le \tilde C_j \int_{\mathbb B_{R_j}(X_j)\cap \Sigma_0}\exp\left\{\frac{3p \alpha_k}{2\pi^2}\int_{\mathbb B_{2R_j}(X_j)\cap \Sigma_0}\log\left(\frac{1}{|X-(\tilde x,0)|}\right)d\mu(\tilde x)\right\} \,dx\\
&\le \tilde C_j \int_{\mathbb B_{R_j}(X_j)} \int_{\mathbb B_{2R_j}(X_j)\cap \Sigma_0} \frac{1}{|X-(\tilde x,0)|^{3p(1-\delta)}}\, d\mu(\tilde x)\, dx\\
&\le C'_j,
\end{align*}
with $C_j'$ depending on $p$, hence on $K$. Summing over $j$ we conclude the proof of \eqref{BM2}.
\end{proof}

Now we are ready for the proof of Theorem \ref{thm1}. First recall the definition of $V_k$ from \eqref{def-vk} and
 set $h_k:=u_k-v_k$ where $u_k:=U_k|_{\R^3}$ and  $v_k:=V_k|_{\R^3}$. We also define $H_k$ by the Poisson representation formula \eqref{kernel-Dirichlet-representation}
 $$H_k(X):=\int_{\R^3}\mathcal K_{\frac{3}{2}}(x-\tilde x,y)h_k(\tilde x)\,d\tilde x, \quad X=(x,y)\in \R^4_+.$$
Notice that, thanks to Proposition \ref{prop:kernel3}, we have $U_k=V_k+H_k$. We now extend $H_k$ on $\R^4_-$ by setting $H_k(x,y):=H_k(x,-y)$. From Lemma  \ref{vk}, Proposition \ref{prop:kernel} and \eqref{eq-1ext} it follows that
 \begin{align*}\label{Hk}
\lim_{y\to0}\frac{\partial}{\partial y}H_k(x,y)=0=\lim_{y\to0}\frac{\partial}{\partial y}\D H_k(x,y) \quad\text{for every }x \in\Sigma_0.
\end{align*}
Therefore,
$$\D^2 H_k=0\quad\text{in }\R^4 \setminus (\Sigma_0^c\times\{0\}),
$$
or, equivalently, $\Delta^2 H_k=0$ in $\Omega$ for every admissible $\Omega\subset\R^4$.
 Since  $$\int_{\R^3}\frac{|v_k(x)|}{1+|x|^6}\,dx\leq C\int_{\Sigma_0}|Q_k(\tilde x)|e^{3u_k(\tilde x)}\int_{\R^3}\frac{\left|\log|x-\tilde x|\right|}{1+|x|^6}\,dxd\tilde x\leq C,$$
 we have
 $$\int_{\R^3}\frac{h_k^+}{1+|x|^6}\,dx\leq \int_{\R^3}\frac{u_k^++|v_k|}{1+|x|^6}\,dx\leq C,$$ thanks to assumption \eqref{extra-assumption}. 
 This implies that,  for $\Omega\Subset\R^4$,
\begin{equation}\label{boundH+}
\int_{\Omega} H_k^+\,dxdy\leq \int_{\R^3}\frac{h_k^+(\tilde x)}{1+|\tilde x|^6}\int_{\Omega}\frac{y^3(1+|\tilde x|^6)}{(|x-\tilde x|^2+y^2)^3}\,dxd\tilde xdy\leq C(\Omega).
\end{equation}

For a given $X_0=(x_0,0)\in \Sigma_0$  we let $R_0>0$ be such that $B_{2R_0}(x_0)\subset \Sigma_0$ and set $$\beta_k=\int_{\B_{R_0}(X_0)}|H_k|\,dxdy.$$

\textbf{Case 1:}  $\beta_k\not\to\infty$. Then, we claim that up to a subsequence,
\begin{equation}\label{convHk}
H_k\to H \quad \text{in }C^\ell(\Omega) \text{ for every }\ell\ge 0, \, \Omega\text{ admissible},
\end{equation}

Indeed, up to a subsequence, $\beta_k\le C$ so that, by elliptic estimates, up to extracting a further subsequence we get $H_k\to H$ in $C^\ell_{\loc}(\B_{R_0}(X_0))$ for every $\ell\ge 0$ and
for a smooth function $H$.
Consider now  $X_1\in \B_{R_0}(X_0)$ and any $R_1>0$ such that $\B_{R_1}(X_1)$ is admissible. Applying Pizzetti formula \eqref{piz} we have
 $$\frac{1}{|\B_{R_1}(X_1)|}\int_{\B_{R_1}(X_1)}H_k(X)\,dX= H_k(X_1)+\frac{R_1^2}{12}\D H_k(X_1).$$
Then, since $|H_k(X_1)|\le C$ and $|\Delta H_k(X_1)|\le C$, the integral on the left-hand side is bounded, and considering \eqref{boundH+} we obtain
$$\int_{\B_{R_1}(X_1)}|H_k|\,dX\le C$$
with a constant depending on $\B_{R_1}(X_1)$. Again by elliptic estimates we have, up to a subsequence, $H_k\to H$ in $C^\ell_{\loc}(\B_{R_1}(X_1))$ for $\ell\ge 0$. Now, in order to prove \eqref{convHk} it suffices to cover the compact set $\overline \Omega$ with a finite number of balls $\B_{R_i}(X_i)$, $i=0,\dots, N$, such that $X_i\in \B_{R_{i-1}}(X_{i-1})$ and use induction to prove that $H_k$ converges in $C^\ell(B_{R_i}(X_i))$ for every $0\le i\le N$ and $\ell\ge 0$.

We now prove that, up to a subsequence, $U_k\to U_\infty$ in $C^{2,\alpha}_{\loc}(\R^4_+\cup(\Sigma_0 \setminus S_1))$ for a function $U_\infty\in C^{2,\alpha}(\R^4_+)$. To show this, consider a point $x_0\in \Sigma_0\setminus S_1$. By the previous discussion and by Lemma \ref{lemmaBM} we have that for $r>0$ sufficiently small
$e^{3u_k}=e^{3v_k}e^{3h_k}$ is uniformly bounded in $L^p(B_r(x_0))$ for some $p>1$. Inserting this into \eqref{def-vk} and taking into account the bound $\|Q_k\|_{L^\infty}\le C$ gives the bound
$$\|v_k\|_{L^\infty(B_\frac{r}{2}(x_0))}\le C,$$
with $C$ independent of $k$.
This in turn implies that $\|u_k\|_{L^\infty(B_\frac{r}{2}(x_0))}\le C.$ By a covering argument, we have proven that $u_k$ is locally uniformly bounded in $\Sigma_0\setminus S_1$.
Inserting this information into \eqref{def-vk} we finally get uniform bounds of the form
\begin{equation}\label{VkC2alpha}
\|V_k\|_{C^{2,\alpha}(\Omega)}\le C(\alpha,\Omega),\quad \alpha\in [0,1),\,\Omega\text{ admissible}, \,S_1\cap \Omega=\emptyset,
\end{equation}
hence by Ascoli's theorem, up to a subsequence, $U_k=V_k+H_k$ converges in $C^{2,\alpha}_{\loc}(\R^4_+\cup(\Sigma_0 \setminus S_1))$.

\medskip

\textbf{Case 2:} If $\beta_k\to\infty$, recalling \eqref{boundH+}, we must have
\begin{align*} \beta_k&\lesssim O(1)+\int_{\B_{R_0}(X_0)}\int_{\R^3}\mathcal{K}(x-\tilde x,y)h_k^-(\tilde x)\,d\tilde xdxdy\\&=O(1)+\frac{4}{\pi^2}\int_{\R^3}h_k^-(\tilde x)\int_{\B_{R_0}(X_0)}\frac{y^3}{(y^2+|x-\tilde x|^2)^3}\,dxdyd\tilde x\\&\approx O(1)+\int_{\R^3}\frac{h_k^-(\tilde x)}{1+|\tilde x|^6}\,d\tilde x\\&\approx \int_{\R^3}\frac{h_k^-(\tilde x)}{1+|\tilde x|^6}\,d\tilde x,
\end{align*}
where for positive sequences $(a_k)$ and $(b_k)$ the notation $a_k\approx b_k$ means $\frac{b_k}{C}\le a_k \le C b_k$ and $a_k\lesssim b_k$ means $a_k\le C b_k$ for a constant $C>0$ not depending on $k$.
Also note that \begin{align*}\beta_k\gtrsim -\int_{\B_{R_0}(X_0)}H_k\,dX&= \int_{\B_{R_0}(X_0)}\int_{\R^3}\mathcal{K}(x-\tilde x,y)(h_k^-(\tilde x)-h_k^+(\tilde x))\,d\tilde xdxdy\\ &=O(1)+\int_{\B_{R_0}(X_0)}\int_{\R^3}\mathcal{K}(x-\tilde x,y)h_k^-(\tilde x)\,d\tilde xdxdy.\end{align*}
Thus
\begin{align}\label{beta-relation}\beta_k\approx \int_{\R^3}\frac{h_k^-}{1+|x|^6}\,dx\approx \int_{\R^3}\frac{h_k}{1+|x|^6}\,dx.\end{align}

By elliptic estimates we have (up to a subsequence) $\frac{H_k}{\beta_k}\to\Phi$ in $C^\ell_{loc}(\B_R(X_0)),$ and using Pizzetti's formula again together with a covering as in Case 1, we also get
$$\frac{H_k}{\beta_k}\to\Phi\quad\text{in }C^\ell(\Omega),\quad \ell\ge 0,$$
for every $\Omega$ admissible, where the bi-harmonic function $\Phi$ satisfies
$$\Phi\leq 0,\quad \int_{\B_R(X_0)}|\Phi|\,dX=1,\quad \partial_y\Phi=\partial_y\D\Phi=0\text{ on }\Sigma_0.$$
We claim that the set $S_\Phi:=\{x\in \Sigma_0:\Phi(x)=0\}$ has empty interior in $\R^3$. Indeed, assume by contradiction that $B_{2\ve}(\xi)\subset S_\Phi$ for some  $\ve>0$ and $\xi\in \Sigma_0$.  Then $\D_x \Phi=0$ on $B_{\ve}(\xi)$. Since $\D^2\Phi=0$, $\Phi\le 0$ and $\Phi\not\equiv 0$, Pizzetti's formula implies that $\partial ^2_{yy}\Phi(\xi,0)<0$ on $S_\Phi$.
In particular, for $r$ small enough
\begin{equation}\label{Phir3}
\int_0^r \Phi(\xi,y)\,dy= \frac{\partial^2_{yy}\Phi(\xi,0)}{6}r^3(1+o(1)),\quad \text{as }r\to 0.
\end{equation}
We write $$H_k(\xi,y) =I_1+I_2, $$ where
$$I_i(y):=\int_{A_i}\mathcal{K}_\frac32(\xi- x,y)h_k( x)\,dx,\quad A_1:=B_\ve(\xi),\, A_2:=\R^3\setminus A_1.$$
Since $h_k=o(\beta_k) $ uniformly on $A_1$ and $\mathcal{K}_{\frac32}\in L^1(\R^4_+)$, we have
$$\int_0^r I_1(y)\,dy=o(\beta_k)\int_0^r\int_{A_1}\mathcal{K}_\frac32 (\xi-x,y)\,dxdy=o(\beta_k).$$ For $x\in A_2$ and $r\leq\ve $ we obtain
\begin{align*}\int_0^r\mathcal{K}_\frac32(\xi-x,y)\,dy&\approx \int_0^r\frac{y^3}{(a^2+y^2)^3}\,dy,\quad a:=|\xi-x|\in [\ve,\infty) \\ &=  \frac{1}{a^2}\int_0^{\frac{r}{a} }\frac{t^3}{(1+t^2)^3}\,dt,\quad\quad  r\leq \ve\leq a,\quad y\mapsto at \\ &\approx \frac{r^4}{a^6}\approx \frac{r^4}{1+|x|^6}.\end{align*}
Therefore, for $r\leq \ve$
$$\int_0^rI_2(y)\,dy\approx r^4\int_{A_2}\frac{h_k(x)}{1+|x|^6}\,dx=O(r^4\beta_k)$$
and thus, for $r\leq\ve $
$$\int_0^r\frac{H_k(\xi,y)}{\beta_k}\,dy=O(r^4), \quad\text{which yields} \quad \int_0^r\Phi(\xi,y)\,dy=O(r^4),$$ contradicting \eqref{Phir3}.


Therefore we have proven that $\Phi|_{\Sigma_0}\not\equiv 0$. Since $\tilde S_\Phi:=\{X\in \R^4:\Phi(X)=0\}$ is analytic, it follows that $S_\Phi=\tilde S_\Phi\cap \Sigma_0$ has Hausdorff dimension at most $2$.

To conclude, we observe that for $\Omega$ admissible and $K\Subset \Omega\setminus \tilde S_\Phi$, we have $\Phi<0$ in $K$, hence $H_k = \beta_k(\Phi+o(1))\to -\infty$ in $K$. In particular $H_k\to-\infty$ locally uniformly on $\Omega\setminus \tilde S_\Phi$ for every $\Omega$ admissible. Writing $e^{3u_k}=e^{3v_k}e^{3h_k}$, one can apply \eqref{def-vk} and Lemma \ref{lemmaBM} as in the last paragraph of Case 1 to prove \eqref{VkC2alpha}. Then we have for $K\Subset \Omega\setminus (S_\Phi\cup S_1)$, $\Omega$ admissible,
$$\frac{U_k}{\beta_k}=\frac{H_k}{\beta_k}+\frac{V_k}{\beta_k}=\Phi+o(1)$$
with $o(1)\to 0$ uniformly on $K$, and we are done. \qed

\begin{rem}It follows from \eqref{beta-relation} that the function $\Phi$  is strictly negative on $\R^4_+$.  Indeed, for a fixed $X_0=(x_0,y_0)\in\R^4_+$,  $$\Phi(X_0)=\lim_{k\to\infty}\frac{1}{\beta_k}\int_{\R^3}\K(x_0-\tilde x,y_0)h_k(\tilde x)\,d\tilde x\approx \lim_{k\to\infty} \frac{-1}{\beta_k}\int_{\R^3}\frac{h_k^-(\tilde x)}{1+|\tilde x|^6}\,d\tilde x\approx -1.$$\end{rem}

\section{Proof of Proposition \ref{prop:existence}}\label{section:existence}

Here we prove the existence result of Proposition \ref{prop:existence}. Given $\Phi\in \mathcal{K}(\Sigma_0)$ set $\phi:=\Phi|_{\Sigma_0}$. We shall first look for solutions $(u_k)$ to
$$(-\D_x)^\frac32u_k=Q_ke^{3u_k}\quad \text{in }\Sigma_0$$
of the form $u_k=v_k+k\phi+c_k$, for some $(v_k)$  bounded in $C^0_{loc}$ and real numbers $c_k=o(k)$.

Since  $\phi<0$ a.e. in $\Sigma_0$, we get $$\lambda_k:=\int_{\Sigma_0}e^{6k\phi}\,dx\xrightarrow{k\to\infty}0.$$  For $\ve>0$ (to be chosen later, independent of $k$) we set $$c_k:=\frac16\log\frac{\ve}{\lambda_k}\quad\text{if }S_\Phi\neq\emptyset\quad\text{and }c_k:=1\quad\text{if }S_\Phi=\emptyset.$$  We claim that $c_k=o(k)$ as $k\to\infty$. When $S_\Phi=\emptyset$ and $c_k=1$ this is obvious. In order to prove the claim also in the case $S_\Phi\ne\emptyset$ we assume, by contradiction, that $c_k\geq 2\delta k$ for some $\delta>0$ and for $k$ large. Then we have $$\ve=\int_{\Sigma_0}e^{6k\phi+6c_k}\,dx\geq \int_{\Sigma_0}e^{6k(\phi+2\delta)}\,dx\geq e^{6k\delta}|\{x\in\Sigma_0:\phi(x)\geq-\delta\}|\to\infty,$$ a contradiction.   

We define $T=T_{\ve,k}:C^0(\bar \Sigma_0)\to \bar \Sigma_0$,  that maps $v\mapsto \bar v$ where
we have set
\begin{align*}\label{barv}
\bar v(x)=\frac{1}{2\pi^2}\int_{\Sigma_0}\log\left(\frac{1}{|x-\tilde x|}\right)Q_k(\tilde x)e^{3(k\phi(\tilde x)+c_k)}e^{3v(\tilde x)}\,d\tilde x,\quad x\in \R^3.
\end{align*}
It follows easily that $T$ is compact. By H\"older's inequality we can now fix $\ve>0$ small enough such that
$$\|T(v)\|_{C^0(\Sigma_0)}\leq  1\quad\text{for  }\|v\|_{C^0(\Sigma_0)}\leq 1.$$
Then Schauder's fixed point theorem implies that $T_{\ve,k}$ has a fixed point, which we call $v_k$. Note that $v_k$ is defined on $\R^3$ and it satisfies
$$(-\D_x)^\frac32 v_k=Q_ke^{3(v_k+k\phi+c_k)}\quad \text{on }\Sigma_0.$$
Let $U_k$ be the extension of $u_k:=v_k+k\phi+c_k$ on $\R^4_+$ using the Poisson formula \eqref{kernel-Dirichlet-representation}. Since $\Phi$ is representable, we can rewrite $U_k$ as   $$U_k=\bar U_k+k\Phi$$ where $\bar U_k$ is the extension of $v_k+c_k$    using the Poisson formula \eqref{kernel-Dirichlet-representation}. It follows that $$\int_{\B_R}|\bar U_k|\, dX=O(c_k)=o(k),$$  and (as in the previous section) $\frac{U_k}{\beta_k}\to\Phi$ with $\beta_k:=k$. 
\qed

\begin{rem}The easiest example of functions $\Phi\in \mathcal{K}(\Sigma_0)$ with $S_\Phi\neq \emptyset$ are polynomials. For instance
$$\Phi(x_1,x_2,x_3,y)=-a_1x_1^2-a_2x_2^2-a_3x_3^2,\quad a_1,a_2,a_3\ge 0,$$
or similar polynomials obtained via translations and rotations.
\end{rem}



\section{Representation formulas and conformal geometry}\label{section:kernel}

In our last section we prove Proposition \ref{prop:kernel}. In addition, we explain the relation to the non-local operator $(-\Delta_x)^{\frac{n}{2}}$ on  $\mathbb R^n$, for $n$ odd. Our ideas come from conformal geometry and could be easily generalized to the curved setting (although there would not be explicit formulas in general).

We remark  that the proofs here involve Fourier transform and are well suited for energy solutions. Since it is not our objective to develop the whole formulation in Sobolev spaces,  but to simply write an explicit representation formula, we assume enough regularity for the statements below. A more precise statement (together with a uniqueness result) was already given in Section \ref{section:representation} for dimension $n=3$. In general, we set $\frac{n}{2}=m+\frac{1}{2}$ for $m\in \mathbb N$, $m\geq 1$.  We note that we are working in a critical dimension and the kernels \eqref{kernel-D} and \eqref{kernel} need to be calculated via analytic continuation as it will be explained below.

Let us first introduce the characterization of the fractional Laplacian $(-\Delta_x)^\gamma$ on $\mathbb R^n$ as a Dirichlet-to-Neumann operator for a higher order extension problem in a non-critical dimension $n>2\gamma$.
We will use the notation $\Delta_b=\Delta_{x,y}+\frac{b}{y}\partial_y$, $b\in\mathbb R$.

\begin{prop}[\cite{Chang-Yang,Case-Chang,Ray}]\label{prop:relation1}
Let $\gamma\in(0,\frac{n}{2})$ be some non-integer. Let also $m <\gamma<m+1$, this is, $m=[\gamma]$, and set
$b(\gamma)=2m + 1-2\gamma$. Assume that $U\in W^{m+1,2}(\mathbb R^{n+1}_+,y^b)$ satisfies the equation
\begin{equation}\label{problem-extension-Ray}\left\{
\begin{split}
&\Delta^{m+1}_b  U=0\quad\text{in }\mathbb R^{n+1}_+,\\
&U=f\quad\text{on }\mathbb R^n,\\
\end{split}\right.
\end{equation}
where $f$ is some function in $H^\gamma(\mathbb R^n)$ and furthermore, that for every positive odd integer $2j + 1 < m + 1$, we have
\begin{equation*}
\lim_{y\to 0} y^b \partial^{2j+1}_y U (x, 0) = 0,
\end{equation*}
and for even integers $2j<m+1$,
\begin{equation*}
\partial^{2j}_y U  (x, 0)  = (\Delta_x)^j U(x, 0)\prod_{l=1}^j
\frac{1}{2\gamma-4(l-1)}.
\end{equation*}
Then we have that
\begin{equation}\label{DtN}
(-\Delta_x)^\gamma f=\tilde{d}_\gamma \lim_{y\to 0}y^b\partial_y \Delta_b^m U(x,y)=:w,
\end{equation}
where we have defined the constant
$$\tilde{d}_\gamma=
\frac{2^{2\gamma}\Gamma(\gamma-m)}{\gamma\, 2^{2m+1}m!\,\Gamma(-\gamma)}.$$
\end{prop}

\begin{rem}
In the proposition above one can take  $\gamma=\frac{n}{2}$ for $n$ odd (via analytic continuation). Then the above proposition yields precisely that the relation between the Dirichlet data \eqref{Dirichlet-condition} and the Neumann-type condition \eqref{Neumann-condition} for problem \eqref{problem-extension} is precisely
\begin{equation*}
w=(-\Delta_x)^{\frac{n}{2}} f \quad \text{in }\mathbb R^n.
\end{equation*}
Note also that if $\gamma=k$ for $k\in\mathbb N$, one also recovers the (entire) powers of the Laplacian using a residue formula at the poles of a meromorphic functional. Since it is not our objective to consider this case, we refer the reader to   \cite{Graham-Zworski:scattering-matrix} for more precise statements.
\end{rem}

The interpretation of Proposition \ref{prop:relation1} comes from conformal geometry, since \eqref{problem-extension-Ray} is the flat version of the extension problem for the construction of the conformal fractional Laplacian $P_\gamma$ on a manifold $M^n$, for $\gamma\in(0,\frac{n}{2})$. $P_\gamma$ is defined as the associated Dirichlet-to-Neumann operator for an extension problem when $M$ is the boundary of a conformally compact Einstein manifold $X$ (or, more generally, asymptotically hyperbolic) and, thus, is a non-local operator on $M$. In the particular case of Euclidean space $\mathbb R^n$ it coincides with $(-\Delta_x)^\gamma$  (compare to \eqref{DtN}). The most important property of $P_\gamma$ is its conformal covariance, this is, if one makes the conformal change of metric
\begin{equation}
\label{change-metric}\tilde g=v^{\frac{4}{n-2\gamma}}g\quad \text{on }M,\quad \text {for some }v>0,
\end{equation}
 then the operator in the new metric can be calculated by the simple intertwining rule
\begin{equation}\label{conformal-property}
P_\gamma^{\tilde g}=v^{-\frac{n+2\gamma}{n-2\gamma}} P_\gamma^g(v \,\cdot).
\end{equation}
We define the fractional curvature of $(M,g)$ as
\begin{equation}\label{defi-Q}
Q^g_\gamma=\frac{1}{\frac{n}{2}-\gamma}P^g_\gamma(1).
\end{equation}
In particular, the conformal property \eqref{conformal-property} yields the $Q_\gamma$ curvature equation
\begin{equation}\label{Q-curvature-equation}
 P_\gamma^g(v)=\left(\tfrac{n}{2}-\gamma \right)Q_\gamma^{\tilde g}\,v^{\frac{n+2\gamma}{n-2\gamma}}\quad \text{on }M.
\end{equation}
If $M$ is the Euclidean space $\mathbb R^n$, this reduces to the fractional Nirenberg equation
\begin{equation*}
(-\Delta_x)^\gamma v=F(x)\,v^{\frac{n+2\gamma}{n-2\gamma}}\quad \text{on }\mathbb R^n.
\end{equation*}

The conformal fractional Laplacian $P_\gamma$ was originally defined in terms of the scattering operator for the conformally compact Einstein manifold $X$. This is inspired in 4-dimensional gravitational Physics (see, for instance, the survey \cite{Gonzalez:survey} and the references therein). Here we will not attempt to give a full description of this geometric problem, instead, we will concentrate on the particular case of Euclidean space  and explain the relation between the scattering problem on hyperbolic space and the higher order extensions for the fractional Laplacian from Proposition \ref{prop:relation1}. The hyperbolic metric is written here as $g^+=\frac{dy^2+|dx|^2}{y^2}$  on the upper half-space $\mathbb R^{n+1}_+$.

\begin{prop}[\cite{Chang-Yang,Case-Chang}]\label{prop:scattering} Fix $\gamma\in(0,\frac{n}{2})\backslash\mathbb N$.
Let $U$ be a solution to \eqref{problem-extension-Ray} with Dirichlet condition $U(\cdot,0)=f$, and   set  $\Phi=y^{\frac{n}{2}-\gamma}U$. Then $\Phi$ is the unique solution of the scattering problem
\begin{equation}\label{prob_sca_gen}\left\{
\begin{split}
\Delta_{g^+}\Phi+\left(\tfrac{n^2}{4}-\gamma^2\right)\Phi&=0\text{ in }\mathbb R^{n+1}_+,
\\
\Phi&=y^{\frac{n}{2}-\gamma}F+y^{\frac{n}{2}+\gamma}G, \qquad F,G\in \mathcal C^{\infty}(\overline{\mathbb R^{n+1}_+}),
\\
F(x,0)&=f\text{ on }\mathbb R^n,
\end{split}\right.
\end{equation}
and it satisfies
\begin{equation}\label{Neumann-scattering}
d_\gamma G(x,0)=w\text{ on }\mathbb R^n,
\end{equation}
where $w$ is the Neumann-type data \eqref{DtN}. Here the constant is given by
\begin{equation*}
d_\gamma=2^{2\gamma}\frac{\Gamma(\gamma)}{\Gamma(-\gamma)}.
\end{equation*}
\end{prop}

The main idea in the proof of our Proposition \ref{prop:kernel} is to obtain a convolution expression for the solution of the scattering problem \eqref{prob_sca_gen} and then to relate it back to the original equation \eqref{problem-extension-Ray} using Proposition \ref{prop:scattering}. Finally, we will use an analytic continuation argument to let $\gamma\to\frac{n}{2}$.\\

Problem \eqref{prob_sca_gen} has been well studied in conformal geometry. For convenience of the reader, we will give full details of the arguments. We first recall Theorems 3.1 and 3.2 in \cite{Chang-Gonzalez}, that relate the scattering problem to a second order Bessel type equation:
\begin{prop}[\cite{Chang-Gonzalez}]
In the notation of Propositions \ref{prop:relation1} and \ref{prop:scattering}, let $\Phi$ be a solution to \eqref{prob_sca_gen} and define, as above,
\begin{equation}\label{U-Phi}
U(x,y):=y^{-\frac{n}{2}+\gamma}\Phi(x,y),
\end{equation}
then $U$ is a solution to the new extension problem
\begin{equation}\label{prob_sca2}\left\{
\begin{split}
\Delta_{x}U +\frac{1-2\gamma}{y}\partial_y U +\partial_{yy}U &=0\text{ in }\R^{n+1}_+,
\\
U(x,0)&=f(x)\text{ in }\R^n.
\end{split}\right.
\end{equation}
Moreover,
\begin{equation*}w=(-\Delta_x)^\gamma f=\frac{d_\gamma}{2\gamma_0} A_m^{-1} \lim_{y\to 0}y^{1-2\gamma_0}\partial_y\left[y^{-1}\partial_y \lp y^{-1}\partial_y\lp\ldots y^{-1}\partial_y U\rp\rp \right],\end{equation*}
where we are taking $m+1$ derivatives in the above expression, $\gamma_0=\gamma-m$, and the constant is given by
\begin{equation*}A_m=2^{m}(\gamma-1)\ldots(\gamma-m+1).\end{equation*}
\end{prop}

Using this Proposition  one can give an explicit expression for the Poisson kernel of the scattering operator in terms of its Dirichlet data:

\begin{prop}\label{prop:Poisson}
Let $\gamma\in(0,\frac{n}{2})\setminus \mathbb N$. Any solution $\Phi$ for \eqref{prob_sca_gen}
can be written as
\begin{equation*}
\Phi(x,y)=\int_{\R^n}\mathcal K_{\gamma}(x-\tilde{x}, y)f(\tilde{x})\,d\tilde{x}
\end{equation*}
 where the kernel is defined by
\begin{equation*}\label{Poisson_kernel}
\mathcal K_\gamma(x,y):=
\kappa_{n,\gamma}\frac{y^{\frac{n}{2}+\gamma}}{(y^2+|x|^2)^{\frac{n}{2}+\gamma}},
\end{equation*}
and the constant is
\begin{equation*}
\kappa_{n,\gamma}=
\frac{\Gamma\left(\frac{n}{2}+\gamma\right)}{\Gamma{(\gamma)}\pi^{\frac{n}{2}}}.
\end{equation*}
\end{prop}

\begin{proof}
This kind of calculation is quite standard by now, but we provide full details for the reader. Let $U$ be as in \eqref{U-Phi}, which is a solution of equation \eqref{prob_sca2}. Take Fourier transform $\hat{\cdot}$ (in the variable $x$) of this equation. Then for any fixed $\xi\in\mathbb R$ we have that $\hat{U}$ satisfies the ODE
$$-|\xi|^2\hat{U}+\tfrac{1-2\gamma}{y}\partial_y\hat{U}+\partial_{yy}\hat{U}=0,$$
which after the change of variable
\begin{equation}\label{yz}
z=|\xi|y,
\end{equation}
becomes
$$\partial_{zz}\hat{U}+\frac{1-2\gamma}{z}\partial_z\hat{U}-\hat{U}=0.$$
This is a Bessel equation.
Lemma \ref{Bessel_cor} implies that  the solution for \eqref{prob_sca2} is given by
\begin{equation}\label{hat-U}
\hat U(\xi,y)=\frac{\Gamma{(\gamma)}^{-1}}{2^{\gamma-1}}\hat f(\xi)|\xi|^{\gamma}y^{\gamma} K_{\gamma}(|\xi|y),
\end{equation}
where $K_\gamma$ is the modified Bessel function of second kind, or Weber's function. Taking inverse Fourier transform we infer
\begin{equation*}
U(x,y)=\frac{\Gamma{(\gamma)}^{-1}}{2^{\gamma-1}(2\pi)^n}\int_{\R^n}\int_{\R^n}
e^{i\xi\cdot (x-\tilde{x})}f(\tilde{x})|\xi|^{\gamma}y^{\gamma} K_{\gamma}(|\xi|y)\,d\tilde{x}d\xi.
\end{equation*}
 This, together with \eqref{U-Phi}, yields
\begin{equation}
\Phi(x,y)=\int_{\R^n}\mathcal K_\gamma(x-\tilde x,y)f(\tilde{x})\,d\tilde{x},
\end{equation}
where we have defined
\begin{equation*}
\mathcal K_\gamma(x,y)=\frac{\Gamma{(\gamma)}^{-1}}{2^{\gamma-1}(2\pi)^n}\,y^{\frac{n}{2}}\int_{\R^n}\cos{(\xi\cdot x)}|\xi|^{\gamma}  K_{\gamma}(|\xi|y)\,d\xi.
\end{equation*}
It is a straightforward computation to check that this $\mathcal K_\gamma(x,y)$ is rotationally invariant (in the variable $x\in\mathbb R^n$).  Thus we can assume, without loss of generality, that $x=|x|e_1$, $e_1\in \mathbb S^n$.

Let us assume first that $n\geq 2$. Using polar coordinates (with $r=|\xi|$) and property \eqref{bessel_2_prop1} in the Appendix, we obtain that
  \begin{equation*}
  \begin{split}
 \mathcal K_\gamma(x,y)=&\frac{\Gamma{(\gamma)}^{-1}}{2^{\gamma-1}(2\pi)^n}|\mathbb S^{n-2}|\,y^{\frac{n}{2}}\int_{0}^{\infty}\int_0^{\pi} e^{i|x|r\cos\theta}r^{n-1+\gamma}(\sin \theta)^{n-2} K_{\gamma}(ry)\,d \theta \,dr
 \\=&\frac{2^{\frac{n}{2}-\gamma}\sqrt{\pi}\,\Gamma\left(\frac{n}{2}\right)}{\Gamma{(\gamma)}(2\pi)^n}|\mathbb S^{n-2}|\,\frac{y^{\frac{n}{2}}}{\,|x|^{\frac{n-2}{2}}\,}\int_{0}^{\infty} r^{\frac{n}{2}+\gamma}J_{\tfrac{n}{2}-1}(|x|r) K_{\gamma}(ry) \,dr.
 \end{split}
 \end{equation*}
Now use  property   \eqref{bessel_3_prop1}
to rewrite this kernel as
\begin{equation*}
\mathcal K_\gamma(x,y)
=\frac{\Gamma\left(\frac{n}{2}-\frac{1}{2}\right)
\Gamma\left(\frac{n}{2}+\gamma\right)}{2\Gamma{(\gamma)}\pi^{n-\frac{1}{2}}}|\mathbb S^{n-2}|\,\frac{y^{\frac{n}{2}+\gamma}}{(y^2+|x|^2)^{\frac{n}{2}+\gamma}}.
\end{equation*}

In the case that $n=1$, using \eqref{bessel_prop1},
\begin{equation*}
\begin{split}
\mathcal K_\gamma(x,y)&=\frac{\Gamma{(\gamma)}^{-1}}{2^{\gamma}\pi}y^{\frac{1}{2}}\int_{\R}e^{i\xi x} {|\xi|^{\gamma}} K_{\gamma}(|\xi|y)\,d\xi\\
&=\frac{\Gamma(\gamma+\frac{1}{2})}{\Gamma(\gamma)\pi\sqrt{\pi}}y^{\frac{1}{2}+\gamma}\int_{\R}e^{i\xi x} \int_0^{\infty}\frac{\cos (|\xi|t)}{(t^2+y^2)^{\frac{1}{2}+\gamma}}\,dt\,d\xi\\
&=\frac{\Gamma(\gamma+\frac{1}{2})}{2\Gamma(\gamma)\pi\sqrt{\pi}}y^{\frac{1}{2}+\gamma}\int_{\R}e^{i\xi x} \int_{\R}\frac{e^{-i\xi t}}{(t^2+y^2)^{\frac{1}{2}+\gamma}}\,dt\,d\xi\\
&=\frac{\Gamma(\gamma+\frac{1}{2})}{\Gamma(\gamma)\sqrt{\pi}}\frac{ y^{\frac{1}{2}+\gamma}}{(x^2+y^2)^{\frac{1}{2}+\gamma}},
\end{split}
\end{equation*}
and this completes the proof of the Proposition.
\end{proof}

\begin{rem}
By looking at the Neumann condition for $\hat U$ given \eqref{hat-U}, recalling the relation \eqref{U-Phi}, one easily obtains \eqref{Neumann-scattering}.
\end{rem}

Let us comment here on the passing to the limit $\gamma\to \frac{n}{2}$  in geometric terms, and the motivation for the $Q$-curvature equation \eqref{Liouville}. This is done by an analytic continuation argument as described in \cite{Graham-Zworski:scattering-matrix,Branson:sharp-inequalities,Branson}.

We write the conformal factor in \eqref{change-metric} as $e^{2u}=v^{\frac{4}{n-2\gamma}}$. Then the fractional curvature  equation \eqref{Q-curvature-equation}  becomes
\begin{equation*}
 \frac{1}{\frac{n-2\gamma}{2}}P_\gamma^g(e^{\frac{n-2\gamma}{2}u})=Q_\gamma^{\tilde g}\,e^{\frac{n+2\gamma}{2}u}\quad \text{on }M.
\end{equation*}
By adding and subtracting a constant (and recalling \eqref{defi-Q}) we obtain
\begin{equation*}
P_\gamma^g\left( \frac{e^{\frac{n-2\gamma}{2}u}-1}{\frac{n-2\gamma}{2}}\right)+Q_\gamma^g
=Q_\gamma^{\tilde g}\,e^{\frac{n+2\gamma}{2}u}\quad \text{on }M.
\end{equation*}
Now we can pass to the limit as $\gamma\to n/2 $, at least formally. We arrive to the non-local Liouville equation
\begin{equation*}
P_{n/2}^g u+Q_{n/2}^g
=Q_{n/2}^{\tilde g}\,e^{nu}\quad \text{in }M,
\end{equation*}
for a change of metric $\tilde g=e^{2u}g$. In the particular case that $\tilde g=e^{2u}|dx|^2$, the background curvature vanishes and the equation reduces to
\begin{equation}\label{equation10}
(-\Delta_x)^{n/2}u=Q_{n/2}^{\tilde g}(x)\,e^{nu}\quad \text{in }\mathbb R^n.
\end{equation}

There is a more general interpretation of \eqref{equation10} (see \cite{Case:boundary-operators} for instance). Indeed, for $n=3$ it is the  $T$-curvature equation on $\mathbb R^3$. The $T$-curvature is defined on the boundary $M^3$ of a smooth 4-manifold $X$, and it was introduced in the setting of functional determinants (\cite{Chang-Qing:zeta1,Chang-Qing:zeta2}, and the survey \cite{Chang:survey}.  The pair of the fourth order $Q$ curvature (the one associated to Paneitz) and the third order $T$ curvature  constitute the higher dimensional analogue of the pair of Gauss curvature and boundary geodesic curvature for surfaces with boundary. Indeed, they are the quantities that appear in the 4-dimensional Gauss-Bonnet formula  for manifolds with boundary (see \cite{Chang-Qing-Yang}).

The $T$ curvature satisfies the following conformal property:  for a conformal change $\tilde g=e^{2u}g$ on $M$,
 \begin{equation*}\label{equation-T}
 P^{g} u +T^{g}=T^{\tilde g}e^{3u}\quad\text{on }M,
 \end{equation*}
where $P$ is a third order boundary operator which generalizes $P_{3/2}$. In the flat case, all these coincide.

\medskip

\begin{proof}[Proof of Proposition \ref{prop:kernel}]
We have just constructed a suitable Poisson kernel to recover the solution $\Phi$ of the scattering problem \eqref{prob_sca_gen} from its Dirichlet data $f$. Using Propositions \ref{prop:relation1} and \ref{prop:scattering} and passing to the limit $\gamma\to\frac{n}{2}$ we obtain \eqref{kernel-Dirichlet} (note that \eqref{problem-extension-Ray} reduces to the original problem \eqref{problem-extension} as $\gamma\to \frac{n}{2}$).

Next, if we wish to  recover $\Phi$ from its Neumann data \eqref{Neumann-scattering}, then it is clear from the symmetry of the equation that, up to the multiplicative constant $d_\gamma$, the only change is $\gamma \leftrightarrow -\gamma$. Note that this duality already appeared in \cite{Caffarelli-Silvestre}. Then we see that the kernel associated to the Neumann condition, using the notation of Proposition \ref{prop:Poisson}, is exactly
\begin{equation*}
\begin{split}
\mathcal K'_\gamma(x,y):=d_\gamma^{-1}\mathcal K_{-\gamma}(x,y)=\frac{\kappa_{n,-\gamma}}{d_\gamma}\,\frac{y^{\frac{n}{2}-\gamma}}
{(y^2+|x|^2)^{\frac{n}{2}-\gamma}}
&=\frac{\Gamma(\frac{n}{2}-\gamma)}{2^{2\gamma}\pi^\frac{n}{2}\Gamma(\gamma)}\,
\frac{y^{\frac{n}{2}-\gamma}}{(y^2+|x|^2)^{\frac{n}{2}-\gamma}}\\
&=: \frac{\tilde\kappa_{n,\gamma}}{\frac{n}{2}-\gamma}
\frac{y^{\frac{n}{2}-\gamma}}{(y^2+|x|^2)^{\frac{n}{2}-\gamma}}.
\end{split}
\end{equation*}
Then, given $w$, there exist $F(x,y)$ and $G(x,y)$ smooth such that a solution to
\begin{equation*}\left\{
\begin{split}
\Delta_{g^+}\Phi+\big(\tfrac{n^2}{4}-\gamma^2\big)\Phi&=0\text{ in }\R^{n+1}_+,
\\
\Phi &= y^{\frac{n}{2}-\gamma} F(x,y)+ y^{\frac{n}{2}+\gamma} G(x,y)\text{ in }\R^n,
\end{split}\right.
\end{equation*}
satisfying $w=d_\gamma G(x,0)$
can be written as
\begin{equation*}
\Phi(x,y)=\int_{\mathbb R^n} \mathcal K'_\gamma(x-\tilde x,y)w(\tilde x)\,dx.
\end{equation*}
Note, however,  that since we have a Neumann problem, $\Phi$ is uniquely defined up to addition of a term of the form $Cy^{\frac{n}{2}-\gamma}$ (this can be seen easily from \eqref{prob_sca2}). We will choose this constant to be able to perform the analytic continuation argument as $\gamma\to\frac{n}{2}$.  Thus we take instead the new kernel
\begin{equation*}
\tilde{\mathcal K}_\gamma(x,y)=\mathcal K'_\gamma(x,y)-\frac{\tilde\kappa_{n,\gamma}}{\frac{n}{2}-\gamma}\, y^{\frac{n}{2}-\gamma}
\end{equation*}
and pass to the limit $\gamma\to n/2$.
Noting that the constant $\tilde\kappa_{n,\gamma}$ extends analytically across $\gamma=n/2$,
and that
$$\lim_{a\rightarrow 0}\left(\frac{1}{az^a}-\frac{1}{a}\right)=-\log  z,$$
we obtain expression \eqref{kernel}. Finally, remark that the $\frac{n}{2}-\gamma$ factor that appears in this proof explains the normalization constant in \eqref{defi-Q}.
\end{proof}

\appendix
\section{Appendix}

\subsection{Pizzetti's formula}

\begin{lem}[\cite{Pizzetti, LM}]
 Let $\D^mh=0$ in $B_{4R}\subset\R^n$. For any $x\in B_R$ and $0<r<R-|x|$ we have
 \begin{align}\label{piz}
  \frac{1}{|B_r|}\int_{B_r(x)}h(z)dz=\sum_{i=0}^{m-1}c_ir^{2i}\D^ih(x),
 \end{align}
 where
  $$c_0=1,\quad c_i=c(i,n)>0,\quad\text{for }i\geq 1.$$ Moreover, for every $k\geq0$ there exists $C=C(k,R)>0$ such that
 \begin{align*}\label{har-est}
  \|h\|_{C^k(B_R)}\leq C\|h\|_{L^1(B_{4R})}.
 \end{align*}
\end{lem}

\subsection{A review of Bessel functions}

Here we summarize some properties of the Bessel functions, mostly taken from \cite{Abramowitz}.

\begin{lem}\label{Bessel_cor}
 Any solution for the Dirichlet problem
  \begin{equation}\label{Bessel_pr}\left\{\begin{split}
&\partial_{zz}\varphi+\tfrac{1-2\gamma}{z}\partial_z\varphi-\varphi=0,\\
&\varphi(+\infty)=0,\\
&\varphi(0)=1,\\
\end{split}\right.
\end{equation}
can be written as
\begin{equation*}\label{sol_bessel_dir}
\varphi(z)=\frac{\Gamma^{-1}{(\gamma)}}{2^{\gamma-1}} z^{\gamma}K_{\gamma}(z).
\end{equation*}
\end{lem}

\begin{proof}
First, rewrite equation \eqref{Bessel_pr} in terms of $\psi(z)=z^{-\gamma}\varphi(z)$  to get $$z^2\partial_{zz}\psi +z\partial_z\psi -(z^2+\gamma^2)\psi=0\text{ in }\R^{n+1}_+.$$
This is a Bessel equation, thus $\psi$ can be written as a linear combination
$$\psi(z)=c_1I_{\gamma}+c_2K_{\gamma},$$
where $I_\gamma$, $K_\gamma$ are the modified Bessel functions of second kind. These have
the following asymptotic behavior:
\begin{equation*}\begin{split}
I_\gamma(z)&\sim \frac{1}{\Gamma(\gamma+1)2^{\gamma}}\,z^\gamma\lp 1+O(z^2)\rp,\\
K_\gamma(z)&\sim \frac{\Gamma(\gamma)}{2^{1-\gamma}}\,z^{-\gamma}
\lp 1+O(z^2)\rp
+\frac{\Gamma(-\gamma)}{2^{\gamma+1}} \,z^\gamma\lp 1+O(z^2)\rp,
\end{split}
\end{equation*}
for $z\to 0^+$. And when $z\rightarrow +\infty$,
\begin{equation*}\begin{split}\label{asymptotic2}I_\gamma(z)\sim \frac{1}{\sqrt{2\pi z}}\,e^z\lp 1+o(1) \rp,\\
K_\gamma(z)\sim \sqrt{\frac{\pi}{2z}}\,e^{-z}\lp1+o(1) \rp.\end{split}\end{equation*}
Finally, the first Dirichlet condition in \eqref{Bessel_pr} implies $c_1=0$, while the second one fixes the value of $c_2=\frac{\Gamma^{-1}{(\gamma)}}{2^{\gamma-1}}$.
\end{proof}

Some useful properties of the modified Bessel functions are
\begin{prop}
\begin{align}\label{bessel_prop1}
 &K_{\gamma}(az)=\frac{\Gamma(\gamma+\frac{1}{2})(2z)^{\gamma}}{\sqrt{\pi}a^{\gamma}}\int_0^{\infty}\frac{\cos (at)}{(t^2+z^2)^{\frac{1}{2}+\gamma}}\,dt,\quad \quad \text{for all }a \text{ with }\ Re(a)>-\tfrac{1}{2}.\\
    \label{bessel_2_prop1}
&J_{\gamma}(z)=\tfrac{z^{\gamma}}{2^{\gamma}\sqrt{\pi}\Gamma(\gamma+\frac{1}{2})}\int_0^{\pi}e^{iz\cos (\theta)}\sin^{2\gamma}(\theta) \,d\theta.\\
\label{bessel_3_prop1}
&\int_0^{\infty}r^{\mu+\nu+1}K_\mu(a r)J_\nu (b r)\,dr=\frac{(2a)^\mu(2b)^\nu\Gamma(\mu+\nu+1)}{(a^2+b^2)^{\mu+\nu+1}}.
\end{align}

\end{prop}

\bigskip
\noindent\textbf{Acknowledgements.}
 A. DelaTorre,  A. Hyder and L. Martinazzi  have been supported by the Swiss National Science Foundation projects n. PP00P2-144669, PP00P2-170588/1 and P2BSP2-172064. The first author is also partially supported by Spanish government grants MTM2014-52402-C3-1-P and MTM2017-85757-P.
M.d.M. Gonz\'alez is supported by Spanish government grants MTM2014-52402-C3-1-P and MTM2017-85757-P, and the Fundaci\'on BBVA grant for  Researchers and Cultural Creators, 2016.

\end{document}